\documentclass{amsart}
\usepackage{graphics}
\usepackage{graphicx}
\usepackage{amsfonts,amsmath,mathrsfs}

\begin{document}

 \newtheorem{thm}{Theorem}[section]
 \newtheorem{coro}[thm]{Corollary}
 \newtheorem{lemma}[thm]{Lemma}{\rm}
 \newtheorem{proposition}[thm]{Proposition}

 \newtheorem{defn}[thm]{Definition}{\rm}
 \newtheorem{ass}[thm]{Assumption}
 \newtheorem{remark}[thm]{Remark}
 \newtheorem{ex}{Example}
\numberwithin{equation}{section}

\newcommand{\bbR}{\mathbb{R}}
\newcommand{\bbC}{\mathbb{C}}
\newcommand{\bbZ}{\mathbb{Z}}

\def\la{\langle}
\def\ra{\rangle}
\def\fac{{\rm !}}

\def\x{\mathbf{x}}
\def\m{\mathbf{m}}
\def\z{\mathbf{x}}
\def\p{\mathbf{p}}
\def\P{\mathbb{P}}
\def\A{\mathbb{A}}
\def\S{\mathbf{S}}
\def\h{\mathbf{h}}
\def\m{\mathbf{m}}
\def\y{\mathbf{y}}
\def\bz{\mathbf{z}}
\def\F{\mathcal{F}}
\def\R{\mathbb{R}}
\def\T{\mathbf{T}}
\def\N{\mathbb{N}}
\def\D{\mathbf{D}}
\def\V{\mathbf{V}}
\def\U{\mathbf{U}}
\def\K{\mathbf{K}}
\def\Q{\mathbf{Q}}
\def\W{\mathbf{W}}
\def\M{\mathbf{M}}
\def\oM{\overline{\mathbf{M}}}
\def\bSigma{\mathbf{\Sigma}}
\def\C{\mathbb{C}}
\def\P{\mathbb{P}}
\def\Z{\mathbb{Z}}
\def\bZ{\mathbf{Z}}
\def\H{\mathcal{H}}
\def\A{\mathbf{A}}
\def\V{\mathbf{V}}
\def\B{\mathbf{B}}
\def\c{\mathbf{C}}
\def\L{\mathcal{L}}
\def\bS{\mathbf{S}}
\def\H{\mathcal{H}}
\def\I{\mathbf{I}}
\def\Y{\mathbf{Y}}
\def\X{\mathbf{X}}
\def\G{\mathbf{G}}
\def\f{\mathbf{f}}
\def\z{\mathbf{z}}
\def\bv{\mathbf{v}}
\def\y{\mathbf{y}}
\def\d{\hat{d}}
\def\x{\mathbf{x}}
\def\b{\mathbf{b}}
\def\c{\mathbf{c}}
\def\bI{\mathbf{I}}

\def\g{\mathbf{g}}
\def\w{\mathbf{w}}
\def\b{\mathbf{b}}
\def\a{\mathbf{a}}
\def\u{\mathbf{u}}
\def\v{\mathbf{v}}
\def\q{\mathbf{q}}
\def\e{\mathbf{e}}
\def\s{\mathcal{S}}
\def\cc{\mathcal{C}}

\def\tg{\tilde{g}}
\def\tx{\tilde{\x}}
\def\tg{\tilde{g}}
\def\tA{\tilde{\A}}

\def\cX{\overline{\mathbf{X}}}
\def\bell{\boldsymbol{\ell}}
\def\bxi{\boldsymbol{\xi}}
\def\balpha{\boldsymbol{\alpha}}
\def\bbeta{\boldsymbol{\beta}}
\def\bgamma{\boldsymbol{\gamma}}
\def\eeta{\boldsymbol{\eta}}
\def\bpsi{\boldsymbol{\psi}}
\def\bsigma{\boldsymbol{\sigma}}
\def\supmu{{\rm supp}\,\mu}
\def\supp{{\rm supp}\,}
\def\cd{\mathcal{C}_d}
\def\cok{\mathcal{C}_{\K}}
\def\vol{{\rm vol}\,}
\def\om{\mathbf{\Omega}}
\def\blambda{\boldsymbol{\lambda}}
\def\btheta{\boldsymbol{\theta}}
\def\bphi{\boldsymbol{\phi}}
\def\bpsi{\boldsymbol{\psi}}
\def\bnu{\boldsymbol{\nu}}
\def\bmu{\boldsymbol{\mu}}
\def\bom{\boldsymbol{\Omega}}
\def\tM{\hat{\M}}
\def\tv{\hat{\v}}
\def\1{\mathbf{1}}
\def\d{\mathrm{d}}
\def\dx{\mathrm{dx}}
\def\dy{\mathrm{dy}}

\title[Distance to Gaussian mixtures]{Gaussian mixtures closest to a given  measure via optimal transport}

\author{Jean B Lasserre}
\address{LAAS-CNRS and Toulouse School of Economics (TSE), BP 54200, 7 Avenue du Colonel Roche, 31031 Toulouse c\'edex 4, France}
\email{lasserre@laas.fr}
\thanks{The author is supported by the AI Interdisciplinary Institute ANITI  funding through the french program
``Investing for the Future PI3A" under the grant agreement number ANR-19-PI3A-0004. This research is also part of the programme DesCartes and is supported by the National Research Foundation, Prime Minister's Office, Singapore under its Campus for Research Excellence and Technological Enterprise (CREATE) programme.}



\begin{abstract}
Given a determinate (multivariate) probability measure $\mu$, we characterize
Gaussian mixtures $\nu_\phi$ which minimize the 
Wasserstein distance $W_2(\mu,\nu_\phi)$ to $\mu$ when 
the mixing probability measure $\phi$ on the 
parameters $(\m,\bSigma)$ of the Gaussians is supported on a compact set $S$.
(i) We first show that such mixtures are optimal solutions
of a particular optimal transport (OT) problem where 
the marginal $\nu_{\phi}$ of the OT problem is also unknown via
the mixing measure variable $\phi$.
Next (ii) by using a well-known specific property of Gaussian measures, 
this optimal transport is then viewed as a Generalized Moment Problem (GMP)
and if the set $S$ of mixture parameters $(\m,\bSigma)$ is a basic compact semi-algebraic set,
we provide a ``mesh-free" numerical scheme 
to approximate as closely as desired 
the optimal distance by solving a hierarchy of semidefinite relaxations of increasing size. In particular,
we neither assume that the mixing measure is finitely supported nor that the variance is the same for all components.
If the original measure $\mu$ is not a Gaussian mixture with parameters $(\m,\bSigma)\in S$,
then a strictly positive distance is detected 
at a finite step of the hierarchy. 
If the original measure $\mu$ is a Gaussian mixture with parameters $(\m,\bSigma)\in S$,
then all semidefinite relaxations of the hierarchy have same zero optimal value. Moreover
if the mixing measure is atomic with finite support, its components can sometimes be extracted from 
an optimal solution at some semidefinite relaxation of the hierarchy when 
Curto \& Fialkow's flatness condition  holds for  some moment matrix.
 \end{abstract}

\maketitle

\section{Introduction}

Comparing mixture distributions (e.g. their ``distance" to each other) is becoming an important topic with many real world applications, and particularly in data science. In addition, in the latter context, for model interpretability 
the mixing measure of components can be as important as the mixture distribution itself.
Quoting \cite{estim-infer}, 
\emph{``standard distances (Hellinger, Total Variation, Wasserstein)  between mixture distributions
do not capture the possibility that similar distributions may arise from mixing completely different 
mixture components, and have therefore different mixing measures"}. 
The relations between mixture distributions and their mixing measures was investigated in \cite{Nguyen}.
So for instance, in the context of \emph{topic models}, in \cite{estim-infer} the authors define what they call the Sketched Wasserstein Distance (SWD) between two mixture distributions, both of which consist of 
a finite mixing of distributions in some set of probability measures on a (Polish) space.
They show that the SWD distance equals the Wasserstein distance between the mixing measures.

Among mixture distributions, Gaussian mixtures  form an important subfamily 
because they can approximate continuous probability densities quite well. 
In particular they are used in statistics for clustering of data and 
to approximate a  large family of distributions of interest in applications; see e.g. \cite{dizioa}, \cite{marron}, \cite{yu}, \cite{mom-var}, \cite{alg-ident}, \cite{wang}, \cite{study}, \cite{Peel},\cite{Moitra}.
Mixtures of Gaussians $\mathcal{N}(\m,\bSigma)$ on $\R^d$ have the well-known and nice property that
every moment $\mu_{\balpha}=\int \x^{\balpha}d\mu$, $\balpha\in\N^d$,
is an \emph{explicit}  polynomial of degree $\vert\balpha\vert$ in the parameters $(\m,\bSigma)$ of the mixture, and therefore
determining whether a real sequence $(y_{\balpha})_{\balpha\in\N^d}$ has a representing 
measure $\mu$ which is some Gaussian mixture, has been recently investigated in e.g. \cite{mom-var,alg-ident} as a specific moment-problem in real analysis. In particular in \cite{alg-ident} the authors prove positive and negative results on 
rational identifiability\footnote{Algebraic identifiability means that there are finitely many (complex) solutions to the moment equations for generic values of the sample moments. On the other hand, rational identifiability
is about generic uniqueness of real solutions, up to the label-swapping action of the symmetric group $S_k$} of $k$-atomic mixing measures of mixture distributions; for instance if $d=1$ then for all
$k$, a $k$-atomic mixing measure can be identified from sufficiently many moments of the mixing distribution
\cite[Theorem 1]{alg-ident}. The same result for mixtures of bivariate Gaussians is a conjecture \cite[Conjecture 2]{alg-ident}; see also \cite{Lindsay} on the key role of moment matrices and determinants in the method of moments.

On the other hand, an important problem  in robust statistics is to estimate parameters of Gaussian mixtures 
from their samples (possibly with noisy data).
In  contributions \cite{valiant-1,kothari-1,kothari-2} from the theoretical computer science community, (theoretical) polynomial time algorithms 
(e.g. sum-of-squares algorithms) have been proposed for efficient learning of mixtures with asymptotic guarantees. In the recent contribution \cite{Yihong}, 
a practical algorithm for optimal estimation of mixtures of finitely many univariate Gaussians with 
same (known or unknown) variance is proposed via  
a (denoised) method of  moments. It combines semidefinite programming and Gauss quadratures to 
estimate a mixture of $k$ univariate Gaussians with same variance. 
In \cite{HK18} the authors consider  the estimator made of mixtures with $k$ atoms (and same variance) 
which minimizes the Kolmogorov distance of its distribution function
to that of the input distribution, and they provide optimal rates of estimation 
(the $k$-atomic mixing distributions are compared with the Wasserstein distance) but no algorithm is provided.  Again, the notion of $k$-idenfiability is of central importance in \cite{HK18}.

In this paper we consider the following problem: Given a probability measure $\mu$ on $\R^d$, and a compact set $S$
of parameters $(\m,\bSigma)$, find a 
mixture $\nu$ of Gaussian measures $\mathcal{N}(\m,\bSigma)$ with 
with parameters $(\m,\bSigma)\in S$, which is the closest to $\mu$.
How close is $\nu$ to $\mu$ is measured e.g. by the $2$-Wasserstein (or Kantorovich) distance $W_2(\mu,\nu)$. That is,
for all $B\in\mathcal{B}(\R^d)$,
\[\nu(B)\,=\,\int_S\left(\frac{1}{(2\pi)^{d/2}\sqrt{\mathrm{det}(\bSigma)}}\int_B\exp(-(\x-\m)^T\bSigma^{-1}(\x-\m)/2)\,\d\x\right)\,\d\phi(\m,\bSigma)\,\]
for some probability $\phi$ on $S$ (the mixing measure of parameters $(\m,\bSigma)\in S$), and 
\[W_2(\mu,\nu)^2\,=\,\inf_{\lambda}\,\{\, \int_{\R^{2d}}\Vert \x-\y\Vert^2\,\d\lambda(\x,\y): \:\lambda_{\x}=\mu\,;\:\lambda_{\y}=\nu\,\}\,,\]
where $\lambda$ is a probability measure on $\R^{2d}$, and $\lambda_{\x}$ (resp. $\lambda_{\y}$) denotes the marginal of $\lambda$ w.r.t. $\x$ (resp. w.r.t. $\y$).
In fact, the results and proposed methodology are also valid 
if one uses the $1$-Wasserstein distance $W_1$ instead of $W_2$. 

\subsection*{Statement of the problem and contribution}
For sake of clarity and simplicity of exposition, we first restrict to the univariate case. Then we briefly describe extension to the multivariate case. While this extension does not pose any theoretical problem, on the other hand the associated numerical scheme is more demanding (simply for question of scalability of the approach).

\paragraph{Statement of the problem} Let $\mathscr{P}(\mathcal{X})$ denote the space of probability measures on a Borel set $\mathcal{X}\subset\R^2$. With $\R_{+}:=\{x: x\geq0\}$, let $S\subset\R\times\R_{+}$ be a set of parameters $(m,\sigma)$ for univariate Gaussian measures
$\mathcal{N}(m,\sigma)$, and let 
$\bmu=(\mu_j)_{j\in\N}$ be the moment sequence of a given probability measure 
$\mu$ on the real line.
The goal is to find a Gaussian mixture $\nu$ with mixing parameters in $S$
that is the closest to $\mu$ with respect to the Wasserstein distance 
\begin{equation}
\label{def-distance-0}
W_2(\mu,\nu)^2\,=\,
\displaystyle\min_{\lambda\in\mathscr{P}(\R^2)}\{\,\int_{\R^2}(x-y)^2\,\d\lambda(x,y):\:
\lambda_x\,=\,\mu\,;\: \lambda_y\,=\,\nu\,\}\,,
\end{equation}
where $\lambda_x$ (resp. $\lambda_y$) is the marginal of $\lambda$ w.r.t. $x$ (resp. w.r.t. $y$) on $\R$. Alternatively one may also use
the Wasserstein distance $W_1(\mu,\nu)=\int \vert x-y\vert\,\d\lambda$ (see Appendix).

As $\nu$ is required to be a Gaussian mixture, it is associated 
with some (not necessarily unique) \emph{mixing} probability  measure
$\phi$ on the set $S$ of Gaussian parameters $(m,\sigma)$, and therefore 
$\nu$ is in fact denoted by $\nu_\phi$,  and reads
\begin{equation}
\label{intro-mixture}
\nu_\phi(B)\,:=\,\int_{S}\left(\frac{1}{\sqrt{2\pi}\sigma}\int_{B} 
\exp(\frac{-(x-m)^2}{2\sigma^2})\,\dx\right)\,\d\phi(m,\sigma)\,,\quad\forall B\in\mathcal{B}(\R)\,.
\end{equation}
Equivalently, $\nu_\phi$ has the density
\[x\mapsto \int_S \frac{1}{\sqrt{2\pi}\sigma}\,
\exp(\frac{-(x-m)^2}{2\sigma^2})\,\d\phi(m,\sigma)\,,\]
w.r.t. Lebesgue measure on $\R$.
Therefore one wishes to solve the optimization problem
\begin{equation}
\label{OT-1}
\tau\,=\,\displaystyle\inf_{\phi\in\mathscr{P}(S)}W_2(\mu,\nu_\phi)^2\,=\,
\displaystyle\inf_{\,\phi\in\mathscr{P}(S),\lambda\in\mathscr{P}(\R^2)}\,\{\,\displaystyle\int (x-y)^2\,\d\lambda(x,y): \:\lambda_x\,=\,\mu\,;\:\lambda_y\,=\,\nu_\phi\}\,.
\end{equation}
Observe that \eqref{OT-1} is an optimal transport problem of a particular type. Indeed 
the second marginal $\lambda_y=\nu_\phi$ of the unknown $\lambda$ is also to be optimized via the
(mixing measure) variable $\phi$ on $S$.

\paragraph{Contribution}

We assume that the set of parameters $S\subset\R\times\R_{+}$ is compact. In contrast to previous works
we do \emph{not} assume that the mixing measure is $k$-atomic (and not even with same variance for all components). Also our algorithm is potentially and directly applicable to mixtures of multivariate Gaussians, although of course its efficiency 
strongly depends on the dimension. At last,
the input probability measure $\mu$ is not necessarily a Gaussian mixture and our
primary goal is to evaluate how far is $\mu$ from a mixture of Gaussians
with parameters $(m,\sigma)$ in a given set $S$. If $\mu$ is indeed such a Gaussian mixture then 
the algorithm helps to detect an associated mixing measure.

I. We first show that if $\mu$ satisfies
\begin{equation}
\label{suff-cond}
\int \exp(c\,\vert x\vert)\,\d\mu(x)\,<\,\infty\,,\end{equation}
for some scalar $c>0$, then \eqref{OT-1} has an optimal solution 
$(\lambda^*,\phi^*)\in\mathscr{P}(\R^2)\times\mathscr{P}(S)$ (i.e., $\tau=W_2(\mu,\nu_{\phi^*})^2$).
Moreover, introducing the moment sequences $\blambda^*=(\lambda^*_{(i,j)})_{(i,j)}$
and $\bphi^*=(\phi^*_{(i,j)})_{(i,j)}$, with
\[\lambda^*_{(i,j)}=\int x^i y^j\,\d\lambda\,,\quad\phi^*_{(i,j)}\,=\,\int m^i \sigma^j\,\d\phi^*\,,\quad\forall (i,j)\in\N^2\,,\]
the couple $(\lambda^*,\phi^*)$ is also  an optimal solution of:
\begin{equation}
\label{OT-moments}
\displaystyle\inf_{\lambda\in\mathscr{P}(\R^2),\phi\in\mathscr{P}(S)}\,\{\,\displaystyle\int (x-y)^2\,\d\lambda: \:
\lambda_{(j,0)}=\mu_j\,;\quad\lambda_{(0,j)}=\displaystyle\int p_j(m,\sigma)\,\d\phi\,,\quad\forall j\in\N\,\}\,,
\end{equation}
which is an \emph{exact  moment-relaxation} of \eqref{OT-1}. To show that \eqref{OT-moments} is equivalent 
to \eqref{OT-1}, one exploits that (i) $S$ is compact,
(ii) the well-known fact that every moment $\mu_j$ of a Gaussian measure $\mu=\mathcal{N}(m,\sigma)$ is an explicit  polynomial $p_j\in\R[m,\sigma]$ of degree $j$, and (iii) that $\mu$ is moment determinate (because of \eqref{suff-cond}). To the best of our knowledge, this is the first characterization of best
Wasserstein-approximations by Gaussian mixtures 
(with parameters in a given set $S$) as \emph{optimal solutions of an optimal transport problem}.

We also obtain that strong duality holds between \eqref{OT-moments} and its dual which reads
\begin{equation}
\label{intro-dual}
\begin{array}{rl}
 \displaystyle\sup_{q\in\R[x],g\in\R[y]}&\{\, \displaystyle\int q\,\d\mu:\:
 q(x)+g(y)\,\leq\,(x-y)^2\,,\:\forall x,y\,;\\
 &\frac{1}{\sqrt{2\pi}\sigma}\displaystyle\int g(x)\,\exp(\frac{-(x-m)^2}{2\sigma^2})\,\dx\,\geq\,0\,,\quad\forall (m,\sigma)\in S\,\}\,,
  \end{array}
\end{equation}
and is very close in spirit to the classical dual of the Monge-Kantorovich optimal transport (with cost $\Vert \x-\y\Vert^2$). 

II. Next, the exact moment formulation \eqref{OT-moments} of
\eqref{OT-1} is a particular instance 
of the ``Generalized Moment Problem" (GMP) (see e.g. \cite{CUP}) whose description is 
trough algebraic data only (because every moment of a Gaussian is a 
\emph{polynomial} in the parameters $(m,\sigma)$).
Therefore one can apply the \emph{Moment-SOS hierarchy} \cite{CUP,HKL} to solve \eqref{OT-moments}. That is,
the optimal value $\tau$ of \eqref{OT-moments} (hence of \eqref{OT-1} as well) can be approximated as closely as desired by solving a sequence (a hierarchy) of semidefinite relaxations of increasing size (as more and more moments are taken into account). 

The degree-$n$ semidefinite relaxation of \eqref{OT-1} (and of \eqref{OT-moments}) is just \eqref{OT-moments} where $\phi\in\mathscr{P}(S)$ and $\lambda\in\mathscr{P}(\R^2)$ are respectively replaced with degree-$2n$ pseudo-moment sequences
$\bphi=\phi_{(i,j)})_{(i,j)\in\N^2_{2n}}$ and $\blambda=(\lambda_{(i,j)})_{(i,j)\in\N^2_{2n}}$, that satisfy necessary semidefinite constraints to be moments of a measure on $S$ and $\R^2$ respectively, coming from Putinar's Positivstellensatz \cite{putinar,CUP}.

If the input  measure is not a mixture of Gaussians 
with parameters $(m,\sigma)\in S$, then the optimal value becomes strictly positive at some step of the hierarchy, 
which provides a certificate that $\mu$ cannot be a mixture of Gaussians with parameters $(m,\sigma)\in S$
(i.e., of the form \eqref{intro-mixture}).

III. On the other hand, if the input  measure $\mu$ \emph{is} a mixture of finitely many 
Gaussian measures with parameters $(m,\sigma)\in S$,  
then $\tau=0$, $\lambda^*=\mu\otimes\mu$, and $\phi^*$ is 
an atomic mixing measure (not necessarily unique) with finite support. If a certain rank condition 
(Curto \& Fialkow's flat extension in \cite[Theorem 3.11]{CUP}) is satisfied at an optimal solution $(\hat{\blambda},\hat{\bphi})$ of some degree-$n$ relaxation in the hierarchy (with optimal value zero), then 
the support and weights of some atomic measure $\hat{\phi}$ on $S$ can be recovered from $\hat{\bphi}$. 
To check whether $\hat{\phi}$ is optimal for \eqref{OT-moments} (and $\hat{\phi}=\phi^*$ and $\phi^*$ is unique)
can be done by checking whether all moments of $\nu_{\hat{\phi}}$ of degree 
higher than $n+1$ match those of $\mu$, i.e., whether
\begin{equation}
\label{check}
\mu_j\,=\,\int p_j(m,\sigma)\,\d\hat{\phi}(m,\sigma)\,,\quad \forall j\,>\,n+1\,.\end{equation}
Checking \eqref{check} for each fixed $j>n+1$ is easy and can be done exactly.

We recall  that identifiability of the mixing measure from moments of the mixture distribution 
is a delicate issue \cite{alg-ident} as in general, several mixing measures can be solutions. However in our setting
we have the additional condition that the mixing mesure is supported on $S$.

Again we emphasize our minimal assumptions: the input measure $\mu$ satisfies 
\eqref{suff-cond} and the parameter set $S$ of admissible mixtures of Gaussians is a compact basic semi-algebraic set.
In particular and in contrast to \cite{Yihong}, the variance $\sigma$ is not fixed and the mixing measures are not 
assumed to be atomic with finite support.

The paper closest in spirit  to ours is the practical algorithm \cite{Yihong} for mixtures $\mu$ of $k$ \emph{univariate} Gaussian measures with \emph{same} variance $\sigma$ (both cases where $\sigma$ is known and unknown are considered in \cite{Yihong}). The author first estimates a vector 
of $2k-1$ moments of $\mu$ via Hermite polynomials, then denoises this vector by projection onto the moment space (via semidefinite programming), and then obtains a resulting $k$-atomic distribution via Gauss quadrature. Nice results in \cite[Theorem 1; (8)]{Yihong} provide optimal rates (with respect to Wasserstein distance $W_1$) provided that $k$ and $\sigma$ are known, and \cite[Theorem 1; (9)-(10)]{Yihong}
if $k$ is known whereas $\sigma$ is unknown.  In \cite{Yihong} 
the semidefinite program is used to "denoise" the input vector of moments by projection onto the moment space. 
The Wasserstein distance is only used to quantify  \emph{\`a posteriori} the error and justify convergence. In our approach, the semidefinite relaxation (i) 
models directly the Wasserstein distance $W_2$  (using $W_1$ is also possible) 
between the input measure and any Gaussian mixture $\nu_\phi$, and (ii)
is parametrized by the number of moments considered.
Finally, notice that the approach in \cite{Yihong} 
is possible thanks to very specific features that are proper to the univariate case only. Namely:

-  (convex) semidefinite programming constraints (exploited in \cite{Yihong}) provide \emph{necessary and sufficient} conditions for a finite real sequence to have a representing measure 
and so the  output of the semidefinite program 
in \cite{Yihong} is a true moment sequence; but similar conditions are only necessary in the multivariate setting.

- similarly, Gauss quadratures also exploited in \cite{Yihong} do not always exist  in the multivariate setting (then called Gauss cubatures); see e.g. \cite{dunkl,cubature-1,cubature-2}.\\

For ease and clarity of exposition, we concentrate in the univariate case but all results of Section \ref{univariate}
are also extended to the multivariate case
which is briefly addressed in Section \ref{multivariate}.

\section{Notation, definitions and preliminary results}

\subsection{Notation and definitions}

Let $\R[x,y]$ denote the ring of real polynomials in the two variables $(x,y)$ and $\R[x,y]_n\subset\R[x,y]$ be its subset 
of polynomials of total degree at most $n$. 
Let $\N^2_n:=\{(i,j)\in\N^2: i+j\vert\leq n\}$
with cardinal $s(n)={n+2\choose 2}$. Let $\v_n(x,y)=(x^i y^j)_{(i,j)\in\N^2_n}$ 
be the vector of monomials up to degree $n$, 
and let $\Sigma[x,y]_n\subset\R[x,y]_{2n}$ be the convex cone of polynomials of total degree at most $2n$ which are sum-of-squares (in short SOS).  A polynomial $p\in\R[x,y]_n$ can be identified with its vector of
coefficients $\p=(p_{(i,j)})\in\R^{s(n)}$ in the monomial basis, and reads
\[(x,y)\quad\mapsto p(x,y)\,:=\,\langle\p,\v_n(x,y)\rangle\,,\quad \forall p\in\R[x,y]\,.\]
With $\mathcal{X}\subset\R^2$, denote by 
$\mathscr{M}(\mathcal{X})_+$ (resp. $\mathscr{C}(\mathcal{X})$), 
the space of positive measures (resp. continuous functions) on $\mathcal{X}$, and by
$\mathscr{P}(\mathcal{X})$, the space of probability measures on $\mathcal{X}$.

For a real symmetric matrix 
$\A=\A^T$, the notation $\A\succeq0$ (resp. $\A\succ0$) stands for $\A$ is positive semidefinite (p.s.d.) (resp. positive definite (p.d.)).
The support of a Borel measure $\mu$ on $\R^2$ is the smallest closed set $A$ such that
$\mu(\R^2\setminus A)=0$, and such a  set $A$ is unique.
A Borel measure with all moments finite is said to be (moment) \emph{determinate} if there is no other measure with same moments.

\paragraph{Riesz functional, moment and localizing matrix.} 
With a real sequence $\bphi=(\phi_{(i,j)})_{(i,j)\in\N^2}$  
(in bold) is associated the \emph{Riesz} linear functional $\phi\in \R[x,y]^*$  (not in bold) defined by
\[p\:(=\sum_{(i,j)}p_{i,j}x^i y^j)\quad \mapsto \phi(p)\,=\,\langle\bphi,\p\rangle\,=\,
\sum_{\balpha}p_{i,j}\,\phi_{(i,j)}\,,\quad\forall p\in\R[x,y]\,,\]
and the moment matrix $\M_n(\bphi)$
with rows and columns indexed by $\N^2_n$ (hence of size $s(n)$), and with entries
\[\M_n(\bphi)((i,j),(i',j'))\,:=\,\phi(x^{i+i'}y^{j+j'})\,=\,\phi_{(i+i',j+j')}\,,\quad (i,j),\,(i',j')\in\N^2_n\,.\]
Similarly, given $g\in\R[x,y]$ ( $(x,y)\mapsto \sum_{(i,j)} g_{i,j}x^i y^j$),
define the new sequence 
\[g\cdot\bphi\,:=\,(\sum_{(k,\ell)} g_{k,\ell}\,\phi_{(i,j)+(k,\ell)})_{(i,j)\in\N^2}\,,\]
and the localizing matrix associated with $\bphi$ and  $g$,
\[\M_n(g\cdot\bphi)((i,j),(i',j'))\,:=\,\sum_{(k,\ell)}g_{k,\ell}\,\phi_{(i+i'+k,j+j'+\ell)}\,,
\quad (i,j),(i',j')\in\N^2_n\,.\]
Equivalently, $\M_n(g\cdot\bphi)$ is the moment matrix associated with the new sequence 
$g\cdot\bphi$. The Riesz linear functional $g\cdot\phi$ associated with
the sequence $g\cdot\bphi$ satisfies
\[g\cdot\phi(p)\,=\,\phi(g\,p)\,,\quad\forall p\in\R[x,y]\,.\]
A  real sequence $\bphi=(\phi_{(i,j)})_{(i,j)\in\N^2}$ has a representing mesure 
if its associated linear functional $\phi$ is a Borel measure on $\R^2$. In this case 
$\M_n(\bphi)\succeq0$ for all $n$; the converse is not true in  general.
In addition, if $\phi$ is supported on  the set $\{\,(x,y)\in\R^2: g(x,y)\geq0\,\}$ then 
$\M_n(g\cdot\bphi)\succeq0$ for all $n$.

\paragraph{Multivariate Carleman condition}
The following condition due to Carleman in the univariate case and later extended by Nussbaum to the multivariate case, is a very useful sufficient condition to ensure that a moment sequence has a representing measure;
see e.g. \cite[Theorem 3.13]{CUP}. We here specialize to the $2$-dimensional case.
\begin{thm}[Bivariate Carleman condition]
\label{th-carleman}
Let $\bphi=(\phi_{(i,j)})_{(i,j)\in\N^2}$ be a real sequence such that $\M_n(\bphi)\succeq0$ for all $n$, and such that
\begin{equation}
\label{th-carleman-1}
\sum_{j=1}^\infty (\phi_{(2j,0)})^{-1/2j}\,=\,+\infty\,;\quad 
\sum_{j=1}^\infty (\phi_{(0,2j)})^{-1/2j}\,=\,+\infty\,.
\end{equation}
Then $\bphi$ has a representing measure $\phi$ on $\R^2$ and $\phi$ is moment determinate.
\end{thm}
For instance, if $\phi$ is a finite Borel measure on $\R^2$ and
$\sup [\int \exp(c\,\vert x\vert)\,\d\phi\,,\,\int \exp(c'\,\vert y\vert)\,\d\,\phi] <\infty$ for some scalars $c,\,c'>0$, then
the moment sequence $\bphi$ satisfies \eqref{th-carleman-1}, and $\phi$ is moment determinate.

\subsection{An intermediate result}
The following result is well-known and is reproduced for sake of clarity.
\begin{proposition}
\label{prop1}
If $\sigma>0$ then for every $j\in\N$, the 
moment
\begin{equation}
\label{def-p_j}
(m,\sigma)\mapsto 
\frac{1}{\sqrt{2\pi}\sigma}\int x^j\,\exp\frac{-(x-m)^2}{2\sigma^2}\,\dx\,,\end{equation}
is a polynomial $p_j\in\R[m,\sigma]$ of total degree at most $j$, and:
\begin{equation}
\label{lem1-1}
p_{2j}(m,\sigma)\,=\,  \sum_{k=0}^j (2k-1)\mathrm{!!}\,\sigma^{2k}\,m^{2(j-k)}\,{2j\choose 2k}\,,\quad\forall j\in\N\,.
\end{equation}
Moreover, if $\sigma=0$ then 
\begin{equation}
\label{lem1-2}
p_{2j}(m,0)\,=\,  m^{2j}\,=\,\int x^{2j}\,\delta_m(\dx)\,,
\quad\forall j\in\N\,.
\end{equation}
\end{proposition}
\begin{proof}
Recall that 
\begin{equation}
\label{ha-1}
\frac{1}{\sqrt{2\pi}\sigma}\int (x-m)^j\,\exp\frac{-(x-m)^2}{2\sigma^2}\,\dx\,=\,\left\{\begin{array}{ll}
0&\mbox{if $j$ is odd}\\
\sigma^j (j-1)\mathrm{!!}&\mbox{if $j$ is even.}\end{array}\right.\,,\quad\forall j\in\N\,,\end{equation}
with for $j\geq2$, $j\mathrm{!!}=j\,(j-2)\,(j-4)\cdots$, $1\mathrm{!!}=1$,
 and the convention $-1\mathrm{!!}=1$.
For instance, $p_0=\1$, $p_1(m,\sigma)=m$, $p_2(m,\sigma)=m^2+\sigma^2$, etc. Next, doing the change of variable $u=(x-m)$ in the integrand of \eqref{lem1-1}, expanding $(u+m)^j$ in the basis of monomials, and summing up, yields \eqref{lem1-1}.
\end{proof}
\begin{remark}
\label{rem-mixture}
(i) A Gaussian mixture is associated with  a (non necessarily unique) mixing probability $\phi\in\mathscr{P}(S)$ 
and in view of \eqref{lem1-2}, $\phi$ may 
tolerate that $\phi(\{\R\times\{0\})>0$, i.e., $\phi$ can mix
Gaussian densities with discrete measures. In other words and with a slight abuse of notation, the Dirac measure $\delta_m$ at point $m$ can be viewed a the degenerate ``Gaussian measure" $\mathcal{N}(m,0)$,
with vector of moments $(m^j)_{j\in\N}=(p_j(m,0))_{j\in\N}$.
For instance if $\mu=\sum_{k=1}^s \gamma_k\,\delta_{x_k}$ for some 
set $\{x_1,\ldots,x_k\}\subset\R$ and scalars $\gamma_k\geq0$, i.e.,
a mixture of $s$ Dirac measures with weights $(\gamma_k)$, then
\[\mu_j\,=\,\int x^j\,d\mu \,=\,\sum_{k=1}^s \gamma_k\,x_k^j\,=\,
\sum_{k=1}^s \gamma_k\,p_j(x_k,0)\,=:\,\int x^j\,\left(\sum_{k=1}^s \gamma_k\,\d\mathcal{N}(x_k,0)\,\right)\,,
\quad\forall j\in\N\,.\]
(ii) So as a consequence, if $S=[-M,M]\times [0,\overline{\sigma}]$ then 
every measure $\mu$ on $[-M,M]$ can be considered a Gaussian mixture where $\mu$ itself is 
the mixing measure. Indeed its moments $(\mu_j)_{j\in\N}$ satisfy
\[\mu_j\,=\,\int m^j\,\d\mu(m)\,=\,\int p_j(m,0)\,\d\mu(m)\,=\,\int \left(\int x^j \d\mathcal{N}(m,0)\right)\,\d\mu(m)\,,\quad j\in \N\,.\]
In particular, every discrete measure on $[-M,M]$ is also a Gaussian mixture
with parameters $(m,0)\in S$. This is not what one usually has in mind
when thinking of Gaussian mixtures, as one would expect a measure $\mu$ with a density w.r.t. 
Lebesgue measure on $\R$.  So this is why one should assume that 
the compact set $S$ satisfies $\sigma\geq\delta>0$ for all $(m,\sigma)\in S$, for some positive scalar $\delta$;
for instance, $S:=[-M,M]\times[\underline{\sigma},\overline{\sigma}]$ with $\underline{\sigma}>0$.
\end{remark}

\begin{coro}
\label{coro1}
Let $\phi$ be a probability measure on $S$. Then with $p_{2j}\in\R[m,\sigma]$, $j\in\N$, as in \eqref{lem1-1}
\begin{equation}
\label{coro1-1}
\sum_{j=1}^\infty \phi(p_{2j})^{-1/2j}\,=\,+\infty\,.
\end{equation}
\end{coro}
\begin{proof}
 Observe that as $S$ is compact, there exists $M>0$ such that
$\vert m\vert, \sigma<M$ for all $(m,\sigma)\in S$, and so
in particular, 
\begin{eqnarray}
\nonumber
p_{2j}(m,\sigma)
\,<\,M^{2j}\,\sum_{k=1}^j\frac{(2j)\mathrm{!}}{(2(j-k))\mathrm{!}}\frac{(2k)\mathrm{!!}}{(2k)\mathrm{!}}&<&
M^{2j}\sum_{k=1}^j\frac{(2j)(2j-1)\cdots (2j-(2k-1))}{(2k-1)\mathrm{!!}}\\
\nonumber
&<&M^{2j}\sum_{k=1}^j(2j)^{2k-1}\,<\,M^{2j}\sum_{k=1}^j(2j)^{2j-1})\\
\label{aux-2j-1}
&<&(2Mj)^{2j}\,,
\end{eqnarray}
and therefore  if $\phi$ is a  probability measure on $S$, then $\phi(p_{2j})<\,(2Mj)^{2j}$ for all $j\in\N$, which in turn implies the desired result
\begin{equation}
\label{aux-2j-2}
\sum_{j=1}^\infty \phi(p_{2j})^{-1/2j}\,>\,\frac{1}{2M}\sum_{j=1}^\infty j^{-1}\,=\,+\infty\,.\end{equation}
\end{proof}

\section{Main result}
\label{univariate}
\subsection{The optimal transport problem \eqref{OT-1} and its exact moment relaxation \eqref{OT-moments}}
Consider the optimal transport problem \eqref{OT-1}.
\begin{thm}
\label{th-OT}
Let $S\subset\R\times\R_+$ be compact, and assume that $\mu\in\mathscr{P}(\R)$ satisfies \eqref{suff-cond}.

(i) The optimal transport problem \eqref{OT-1} has an optimal solution $(\phi^*,\lambda^*)\in\mathscr{P}(S)\times\mathscr{P}(\R^2)$ which is also an optimal solution of \eqref{OT-moments}. Moreover, both measures
$\lambda^*\in\mathscr{P}(\R^2)$ and $\nu_{\phi^*}\in\mathscr{P}(\R)$ are moment determinate.

(ii) Moreover, $\tau=0$ if and only if $\lambda^*=\mu\otimes\mu
$ and $\mu=\nu_{\phi^*}$,  i.e., $\mu$ is a Gaussian mixture with $\phi^*$ a mixing measure of parameters 
$(m,\sigma)\in S$.
\end{thm}
For clarity of exposition a proof is postponed to Section \ref{appendix}.
 
\begin{remark}
(a) Notice that the mixing probability measure $\phi^*\in\mathscr{P}(S)$  is not necessary unique. That is,
two different mixing measures $\phi_1$ and $\phi_2$ may produce the same mixture distribution $\nu_{\phi_1}=\nu_{\phi_2}$. This uniqueness issue is related to rational identifiability issue already mentioned and explored in e.g.
\cite{mom-var,alg-ident}. However in our restricted setting, 
uniqueness is perhaps easier to get as the support of the mixing measure  is \emph{not} the whole space $\R^2$ but a 
compact set $S\subset\R^2$.

(b) In Theorem \ref{th-OT}, 
 $\nu_{\phi^*}$ is  a mixture of Gaussian measures with parameters $(m,\sigma)\in S$.
 If $\sigma=0$ is tolerated in $(m,\sigma)\in S$, the mixture $\phi^*$ can be made of ``pure" Gaussian measures $\mathcal{N}(m,\sigma)$ with $\sigma>0$ and 
 atomic measures $\delta_{m}``="\mathcal{N}(m,0)$. If one wishes to 
 obtain the closest mixture $\nu_{\phi^*}$ of ``pure" Gaussian measures $\mathcal{N}(m,\sigma)$ with $\sigma>0$, (i.e., with no atomic part), then  in Theorem \ref{th-OT} 
  one should replace 
 $S\subset\R\times\R_{+}$ with $S\subset\R\times\R_{++}$ (with $\R_{++}:=\{x: x>0\}$). 
 As $S$ is assumed to be compact this implies that for some $\delta>0$,
 $\sigma\geq\delta$ for all $(m,\sigma)\in S$.
 
The interesting case is precisely when $\sigma=0$ is \emph{not} tolerated. Indeed if $\sigma=0$ is tolerated then any probability measure 
 $\mu$ supported on the set $\{\,m: (m,0)\in S\,\}$ (in particular atomic measures) is the ``Gaussian mixture" $\mathcal{N}(m,0)\,\d\mu(m)$ with mixing measure $\mu$ itself,
 which is not really what one wants to detect. see  Remark \ref{rem-mixture}(ii). 
 \end{remark}

\paragraph{A dual of \eqref{OT-1}}

For any $g\in\R[y]$ write $y\mapsto g(y):=\sum_k g_ky^k$ where $(g_k)$ is the vector of coefficients of $g$
in the monomial basis $(y^k)_{k\in\N}$. Consider the optimization problem:
\begin{equation}
 \label{OT-dual}
 \begin{array}{rl}
 \tau^*=\displaystyle\sup_{q\in\R[x],g\in\R[y]}&\{\, \displaystyle\int q\,\d\mu:
 \:q(x)+g(y)\,\leq\,(x-y)^2\,\quad\forall x,y\in\R\,;\\
  &\displaystyle\sum_k g_k\,p_k(m,\sigma)\,\geq\,0\,,\quad\forall (m,\sigma)\in S\,\}\,.
 \end{array}
\end{equation}
Observe that:
\[\sum_k g_k\,p_k(m,\sigma)\,\geq\,0\,,\:\forall (m,\sigma)\in S\quad\Leftrightarrow\quad
\frac{1}{\sqrt{2\pi}\sigma}\int g(x)\,\exp(\frac{-(x-m)^2}{2\sigma^2})\,\dx\,\geq\,0\,,\]
for all $(m,\sigma)\in S$.

\begin{proposition}
 The optimization problem \eqref{OT-dual} is a dual of \eqref{OT-1}, i.e., weak duality $\tau\geq\tau^*$ holds.
  \end{proposition}
 
\begin{proof}
 Let $(\lambda,\phi)$ (resp. $(q, g)$) be a feasible solution of \eqref{OT-1}
 (resp. \eqref{OT-dual}). Then as $\lambda_x=\mu$ and
 $\lambda_y=\nu_\phi$,
 \begin{eqnarray*}
 \int (x-y)^2\,\d\lambda(x,y)\,\geq\,\int (q+g)\,\d\lambda
&=&\int q\,d\mu+\int g\,\d\nu_\phi\\
&=&\int q\,\d\mu+\int_S\underbrace{(\sum_k g_k\,p_k)}_{\geq0\mbox{ on $S$}}\,\d\phi\geq\int q\,\d\mu\,,
 \end{eqnarray*}
 and as $(\lambda,\phi)$ and $(q, g)$ are arbitrary feasible solutions,
 it follows that $\tau\geq\tau^*$.
\end{proof}
\subsection{A hierarchy of semidefinite relaxations}

We here consider the case where the set $S\subset\R^2$ of parameters $(m,\sigma)$ is the compact basic semi-algebraic set
\begin{equation}
\label{def-S}
 S\:=\,\{\,(m,\sigma):\: u_j(m,\sigma)\,\geq0\,,\:j=1,\ldots,s\,\}\,,
\end{equation}
for some polynomials $u_j\subset \R[m,\sigma]$, $j=1,\ldots,s$, and we let $u_0:=\1$ (the constant polynomial equal to $1$ for all $(m,\sigma)$.
Moreover as $S$ is compact, we also assume that we know a scalar $R$ such that
$S\subset \{(m,\sigma):  m^2+\sigma^2<R^2\}$ and without changing $S$ we include the redundant quadratic constraint 
$R^2-m^2-\sigma^2\geq0$ in its definition \eqref{def-S},
with for instance $u_1(m,\sigma)=R^2-m^2-\sigma^2$. 

Next, let $d_j:=\lceil\mathrm{deg}(u_j)/2\rceil$, $n_0:=\max_j d_j$ and
fix $n\geq n_0$. With $p_j\in\R[m,\sigma]$ as in \eqref{def-p_j}, define:
\begin{equation}
\label{relax-mom}
\begin{array}{rl}\tau_n\,=\,\displaystyle\min_{\bphi,\blambda}&\{\,\lambda((x-y)^2)\,:
\quad\lambda_{(j,0)}=\mu_j\,;\quad
\lambda_{(0,j)}-\phi(p_j(m,\sigma))\,=\,0\,,\quad\forall j\leq 2n\,;\\
&\M_n(\blambda)\,\succeq0\,,\,\M_n(\bphi)\,\succeq\,0\,,\:\M_{n-d_j}(u_j\cdot \bphi)\,\succeq\,0\,,\quad j=0,\ldots,s\,\}\,,
\end{array}\end{equation}
where $\blambda=(\lambda_{(i,j)})_{(i,j)\in\N^2_{2n}}$ and
$\bphi=(\phi_{(i,j)})_{(i,j)\in\N^2_{2n}}$.
Problem \eqref{relax-mom} is a semidefinite program\footnote{A semidefinite program is a convex conic program on the cone of positive semidefinite matrices. Up to arbitrary (but fixed) precision, 
it can be solved efficiently; see e.g. \cite{anjos-lasserre,odonnell}}. Its dual reads:
\begin{equation}
\label{relax-sos}
\begin{array}{rl}\tau^*_n\,=\,\displaystyle\sup_{q,g,\sigma,\theta_j}&\{\,\displaystyle\int q\,d\mu:\quad
q(x)+g(y)+\sigma(x,y)\,=\,(x-y)^2\,,\quad\forall x\,,y\in\R\,;\\
&\displaystyle\sum_{k=0}^{2n} g_k\,p_k(m,\sigma)\,=\,\displaystyle\sum_{j=0}^s \theta_j(m,\sigma)\,u_j(m,\sigma)\,;\\
&q\in\R[x]_{2n}\,,g\in\R[y]_{2n}\,;\:\sigma\in\Sigma[x,y]_n\,;\:\theta_j\in\Sigma[m,\sigma]_{n-d_j}\,,\:j=0,\ldots,s\,\}\,,
\end{array}\end{equation}
with $\tau^*_n\leq\tau_n$ for all $n\geq n_0$.
\begin{lemma}
\label{lem1}
For each fixed $n\geq n_0$, \eqref{relax-mom} is a semidefinite program and a convex relaxation of the infinite-dimensional problem \eqref{OT-1} and so $\tau_n\leq \tau$ for all $n\geq n_0$.  Moreover, if $S$ has nonempty interior and 
$\mathrm{supp}(\mu)$ contains an open set, then $\tau_n=\tau^*_n$ and \eqref{relax-sos} has an optimal solution
$(q^*,g^*,\theta^*_0,\ldots,\theta^*_s)$.
\end{lemma}
\begin{proof}
 Let $(\lambda,\phi)\in\mathscr{P}(\R^2)\times\mathscr{P}(S)$ be a feasible solution of \eqref{OT-1}, and let 
 $\blambda=(\lambda_{(i,j)})_{(i,j)\in\N^2_{2n}}$ and
$\bphi=(\phi_{(i,j)})_{(i,j)\in\N^2_{2n}}$ be the vectors
of degree-$2n$ moments of $\lambda$ and $\phi$ respectively. Then  the couple $(\blambda,\bphi)$ is a feasible solution of \eqref{relax-mom}, and so $\tau_n\leq\tau$ for all $n\geq n_0$. Next, let $\phi$ be the probability measure uniformly distributed on $S$, and let $\lambda:=\mu\otimes\nu_{\phi}$. Then as $S$ has nonempty interior, 
$\M_n(u_j\cdot\bphi)\succ0$ for all $j=0,\ldots,s$, and $\M_n(\blambda)\succ0$. Indeed, suppose that for some
$h\in\R[x,y]_n$ with coefficient vector $\h$, 
\begin{eqnarray*}
0&=&\langle \h,\M_n(\blambda)\,\h\rangle\\
&=&\int h(x,y)^2\,\d\lambda(x,y)\\
&=&\int_{\R} \left(\int_{\R}h(x,y)^2\,\d\nu_{\phi}(y)\right)\,\d\mu(x)\\
&=&\int_{\R}\left(\int_{S}\frac{1}{\sqrt{2\pi}\sigma}\int_{\R} h (x,y)^2\exp(-(y-m)^2/2\sigma^2)\,\dy\,\d\phi(m,\sigma)\right)\,\d\mu(x)\,.
\end{eqnarray*}
We next prove that then $h=0$ and so $\M_n(\blambda)\succ0$. Observe that with $h\in\R[x,y]_n$, one may write
\[h(x,y)^2\,=\,\sum_{k=0}^{2n} \theta^h_{n-k}(x)\,y^{k}\,,\quad\mbox{with $\theta^h_{n-k}\in\R[x]_{2n-k}$ for all $k=0,\ldots,2n$,}\]
and therefore 
\[\frac{1}{\sqrt{2\pi}\sigma}\int_{\R} h(x,y)^2\,\exp(-(y-m)^2/2\sigma^2)\,\dy\,=:\,
\sum_{k=0}^{2n}\theta^h_{n-k}(x)\,p_k(m,\sigma)\,=:\,
q_h(x,m,\sigma)\,,\]
is a polynomial in $\R[x,m,\sigma]_{2n}$. Moreover, for all $x\in\R$,
\[q_h(x,m,\sigma)\,\geq\,
\left(\frac{1}{\sqrt{2\pi}\sigma}\int_{\R} \vert h(x,y)\vert \,\exp(-(y-m)^2/2\sigma^2)\,\dy\right)^2\,\geq\,0\,,\quad\forall (m,\sigma)\in\,S\,.\]
Hence,
\begin{eqnarray*}
0&=&\int_{\R}\int_{S}\frac{1}{\sqrt{2\pi}\sigma}\int_{\R} h (x,y)^2\exp(-(y-m)^2/2\sigma^2)\,\dy\,\d\phi(m,\sigma)\,\d\mu(x)\\
&=&\int_{\R}\int_S q_h(x,m,\sigma)\,\d\phi(m,\sigma)\,\d\mu(x)\,,
\end{eqnarray*}
implies that $q_h(x,m,\sigma)=0$, for $\mu\otimes\phi$-a.e. $(x,m,\sigma)\in \R\times S$. 
As $S$ has nonempty interior, $\mathrm{supp}(\mu)$ contains an open set, and 
$q_h$ is a polynomial, this implies $q_h\equiv 0$. But then this in  turn implies $h(x,y)=0$ for all $x,y$,
and therefore $h\equiv 0$.  Hence the couple $(\blambda,\bphi)$ is a strictly feasible solution of \eqref{relax-mom},
that is,  Slater's condition\footnote{Slater condition holds for the finite-dimensional conic program 
$\displaystyle\min_\x\{\,\c^T\x: \A\x=\b\,;\:\x\in\,K\,\}$ for a linear mapping $A:\R^p\to\R^q$, vectors $\c\in\R^p$,
$\b\in\R^q$, and a convex cone
$K\subset\R^p$, if there exists an admissible solution $\x_0\in\mathrm{int}(K)$.} holds for \eqref{relax-mom}. This in turn implies that there is not duality gap
between \eqref{relax-mom} and its dual \eqref{relax-sos}, i.e., $\tau_n=\tau^*_n$,
and as $\tau_n\geq0$, their value is finite.
\end{proof}
\begin{thm}
\label{th-hierarchy}
 Let $S\subset\R\times \R_{+}$ as in \eqref{def-S} be compact, and 
 let $\mu\in\mathscr{P}(\R)$ be a probability measure such that \eqref{suff-cond}
 holds for some scalar $c>0$.  
 
 (i) For every fixed $n$, \eqref{relax-mom} is a semidefinite relaxation of \eqref{OT-moments} (hence of \eqref{OT-1}) and has an optimal solution 
 $(\blambda^{(n)},\bphi^{(n)})$ with associated optimal value $\tau_n\leq \tau$ for all $n\geq n_0$.
 
 (ii) For any accumulation point $(\blambda^*,\bphi^*)$ of the sequence $(\blambda^{(n)},\bphi^{(n)})_{n\in\N}$ of optimal moment-sequences $(\blambda^{(n)},\bphi^{(n)})$ of \eqref{relax-mom}, 
 $\blambda^*$ (resp. $\bphi^*$) has a determinate representing measure $\lambda^*$ on $\R^2$ (resp. $\phi^*$ on $S$) and
 $(\phi^*,\lambda^*)$ is an optimal solution of \eqref{OT-1} and \eqref{OT-moments}. That is:
 \[\tau_n\uparrow \tau \,=\,W_2(\mu-\nu_{\phi^*})^2\quad\mbox{as $n\to\infty$}\,.\]
 \end{thm}
 
For clarity of exposition a proof is postponed to Section \ref{appendix}.

\begin{coro}
\label{cor1}
Let $\tau_n$ and $\tau^*_n$ be as in \eqref{relax-mom} and \eqref{relax-sos}, respectively.
Under the assumption in Theorem \ref{th-hierarchy} and if $S$ has nonempty interior and $\mathrm{supp}(\mu)$ contains an open set, 
then $\tau=\lim_{n\to\infty}\tau_n=\lim_{n\to\infty}\tau^*_n$, and therefore 
there is no duality gap between \eqref{OT-moments} and \eqref{OT-dual}, that is,
\begin{equation}
 \label{OT-dual-new}
 \begin{array}{l}
\displaystyle\inf_{\lambda\in\mathscr{P}(\R^2),\phi\in\mathscr{P}(S)}\:\{\,\displaystyle\int (x-y)^2\,\d\lambda: \\
\mbox{s.t.}\quad \lambda_{j0}=\mu_j\,,\:\forall j\in\N\\
\lambda_{0,j}-\displaystyle\int p_j(m,\sigma)\,\d\phi\,=\,0\,,\:\forall j\in\N\,\}
\end{array}\,=\,\quad
 \begin{array}{l}
 \displaystyle\sup_{q\in\R[x],g\in\R[y]}\:\{\, \displaystyle\int q\,\d\mu:\\
 \mbox{s.t.}\: q(x)+g(y)\,\leq\,(x-y)^2\,\:\forall x,y\in\R\,;\\
 \displaystyle\sum_k g_k\,p_k(m,\sigma)\,\geq\,0\,,\quad\forall (m,\sigma)\in S\,\}\,.
 \end{array}
\end{equation}
\end{coro}
\begin{proof}
 By Lemma \eqref{lem1}, $\tau_n=\tau^*_n$ for all $n\geq n_0$, and by Theorem 
 \ref{th-hierarchy}, $\tau=\lim_{n\to\infty}\tau_n$. Therefore $\tau^*$ in \eqref{OT-dual}
 is equal to $\tau$, which yields \eqref{OT-dual-new}.
\end{proof}
Observe that \eqref{OT-dual-new} resembles the usual duality in optimal transport
when both marginals $\lambda_x$ and $\lambda_y$ are fixed; here the marginal $\lambda_y$ 
is also part of the optimization via the mixing measure $\phi$.

\begin{coro}
\label{cor2}
Let $S\subset\R\times \R_+$ be compact with nonempty interior and let
$\mu\in\mathscr{P}(\R)$ be such that \eqref{suff-cond} holds for some scalar $c>0$
and its support contains an open set. Then
$\mu$ is a mixture of Gaussians
with parameters $(m,\sigma)\in S$ if and only if 
for every $n\geq n_0$, $(q^*,g^*)=(0,0)$ and $\theta^*_j=0$ for all $j=0\ldots,s$,
is an optimal solution of \eqref{relax-sos}.
\end{coro}
\begin{proof}
It $\mu$ is a Gaussian mixture with mixing measure $\phi^*\in\mathscr{P}(S)$,
then $\tau=W_2(\mu,\nu_{\phi^*})=0$.  As $0\leq\tau_n\leq\tau=0$ one obtains $\tau_n=\tau^*_n=0$ for all
$n\geq n_0$ and $(q,g)=(0,0)$ with $\theta^*_j=0$ for all $j=0,\ldots,s$, is an obvious optimal solution
of \eqref{relax-sos}.
\end{proof}

\subsection{Recognizing a Gaussian mixture}
As a consequence of Corollary \ref{cor2}, if the input probability measure $\mu$
is \emph{not} a mixture of Gaussian measures with parameters $(m,\sigma)\in S$, then the optimal value 
of \eqref{relax-mom} becomes positive at some step $n_*\geq n_0$ 
and then remains positive for all $n\ge n_*$ (as the sequence is monotone non decreasing). So the 
sequence of optimal values $(\tau_n)_{n\in\N}$  of the hierarchy \eqref{relax-mom} 
permits to detect in finitely many steps if $\mu$ is not a Gaussian mixture (with parameters $(m,\sigma)\in S$).

Recall that
$d_j:=\lceil\mathrm{deg}(u_j)/2\rceil$ and let $v:=\max_jd_j$.

\begin{thm}
\label{mu-mixture}
With $S\subset\R^2$ as in \eqref{def-S}, let $\mu\in\mathscr{P}(\R)$ be a given probability measure 
with finite moments $\bmu=(\mu_{(i,j)})_{(i;j)\in\N^2}$, and let $\tau$ and $\tau_n$ be as in \eqref{OT-1} and \eqref{relax-mom} respectively. 

(i) $\mu$ is a mixture of Gaussian measures, all with parameters $(m,\sigma)\in S$, if and only if 
$(\mu\otimes\mu,\phi^*)$ is  an optimal solution of \eqref{OT-1} for some $\phi^*\in \mathscr{P}(S)$.
Moreover, $\tau_{n}=\tau=0$ for all $n\geq n_0$, i.e., the optimal value $0$ is obtained at every step 
of the hierarchy of semidefinite relaxations \eqref{relax-mom}.

In addition, if $\mu$ is a mixture of finitely many (say $r$) Gaussian measures, all with parameters $(m,\sigma)\in S$, 
then for every $n$ sufficiently large, the corresponding degree-$2n$ vector of moments $(\blambda^*,\bphi^*)$
respectively associated with $\mu\otimes\mu$ and $\phi^*$, is an optimal solution of \eqref{relax-mom} and
\begin{equation}
\label{curto}
\mathrm{rank}(\M_n(\bphi^*))\,=\,\mathrm{rank}(\M_{n-v}(\bphi^*))\,=\,r\,.\end{equation}

(ii) Conversely, let $(\blambda^*,\bphi^*)$ be an optimal solution of some degree-$2n$ relaxation
\eqref{relax-mom} with $\tau_n=0$, and suppose that \eqref{curto} holds.  Then $\bphi^*$ is 
the degree-$2n$  moment vector of some $r$-atomic probability measure $\phi^*$ on $S$. Moreover,
$\mu=\nu_{\phi^*}$ (i.e. $\mu$  is a Gaussian mixture with mixing measure $\phi^*$)  
if and only if 
\begin{equation}
\label{check-1}
\mu_j\,=\,\int p_j(m,\sigma)\,\d\phi^*\,,\quad \forall j\,>\,n+1\,.\end{equation}
\end{thm}
\begin{proof}
(i) \emph{ Only if part:} By definition there exists $\phi^*\in\mathscr{P}(S)$ such that
 \[\mu(B)\,=\,\int_S\left(\frac{1}{\sqrt{2\pi}\sigma}\,\int_B \exp(\frac{-(x-m)^2}{2\sigma^2})\,\dx\right)\,\d\phi^*(m,\sigma)\,,\quad\forall B\in \mathcal{B}(S)\,.\]
 Then $\tau=W_2(\mu,\nu_{\phi^*})=0$, and with $\lambda^*:=\mu\otimes\mu$,
 the couple $(\lambda^*,\phi^*)$ is an obvious optimal solution of \eqref{OT-1}. Moreover, $\tau_n=0$ for all $n$, follows from $0\leq\tau_n\leq\tau$ and $\tau=0$.

\emph{If part:} If $(\mu\otimes\mu,\phi^*)$ is an optimal solution of \eqref{OT-1} then 
$\mu=\lambda_y=\nu_{\phi^*}$, i.e., $\mu$ is a Gaussian mixture with mixing measure $\phi^*\in\mathscr{P}(S)$, and $\tau=0=W_2(\mu,\nu_{\phi^*})^2$.
Next, fix $n\geq n_0$ arbitrary. The finite vector of degree-$2n$ moments $(\blambda^*,\bphi^*)$ of $\lambda^*=\mu\otimes\mu$ and $\phi^*$ respectively, is an obvious 
feasible solution of \eqref{relax-mom}. Moreover $\lambda^*((x-y)^2)=\int (x-y)^2\d\mu(x)\,\d\mu(y)=0$,
and as $\tau_{n}\geq0$, $(\blambda^*,\bphi^*)$ is an optimal solution of \eqref{relax-mom} with $\tau_n=0$.

Next, as $\phi^*$ is $r$-atomic, $\mathrm{rank}(\M_n(\bphi^*))=r$ for all sufficiently large $n$. As 
$\M_{n-v}(\bphi^*)$ is a submatrix of $\M_n(\bphi^*)$, \eqref{curto} follows. 

(ii) Conversely, if \eqref{check-1} holds at an optimal solution
of a degree-$2n$ relaxation \eqref{relax-mom},
then by Curto \& Fialkow's flat extension theorem \cite[Theorem 2.47]{CUP-2}, $\bphi^*$ is the degree-$2n$ moment sequence of 
some $r$-atomic $\phi^*\in\mathscr{P}(S)$. Next, $\tau_n=0$ implies
$\lambda^*((x-y)^2)=0$ and so the vector $\p\in\R^{{2+n\choose 2}}$ of coefficients of the polynomial 
$(x,y)\mapsto p(x,y):=(x-y)\in\R[x,y]_n$ is in the kernel of $\M_n(\blambda^*)$ as
\[\langle \p,\M_n(\blambda^*)\,\p\rangle\,=\,\lambda^*(p^2)\,=\,\lambda^*(x-y)^2)\,=\,0\,.\]
That is, the second and third columns of $\M_n(\blambda^*)$ (respectively indexed by the monomials 
$x$ and $y$) are identical. In particular, this implies $\lambda^*_{(j,0)}=\lambda^*_{(0,j)}$ for all $j=0,\ldots n+1$.
Equivalently $\mu_j=(\nu_{\phi^*})_j$ for all $j\leq n+1$, and therefore as $\mu$ is determinate,
$\mu=\nu_{\phi^*}$ (and so $W_2(\mu,\nu_{\phi^*})=0$) if only if $\mu_j=\lambda^*_{(0,j)}$ for all $j$, and so if and only if \eqref{check-1} holds. 
\end{proof}
 
  The sufficient Curto \& Fialkow's flatness condition \eqref{curto} in Theorem \ref{mu-mixture} is very useful to detect whether $\mu$ is a
  Gaussian mixture $\nu_{\phi^*}$ with an $r$-atomic mixing measure $\phi^*$ on $S$, in solving \emph{finitely many} 
  semidefinite relaxations. Indeed if \eqref{curto} holds then it remains to check whether \eqref{check-1} holds (with no optimization involved).
 
\begin{ex}
 \label{mu-not-mixture}
 Let $S=[.07,1]\times [.02,1]$ and $\mu=r*\mathcal{N}(.1,.2)+(1-r)*\mathcal{N}(.5,.5)$ with $r\in (0,1)$.
 Then with $r=.2$ or $r=.3$, the atomic measure $\phi^*=r*\delta_{(.1,.2)}+(1-r)*\delta{(.5,.5)}$
 is detected at step $n=6$ of the semidefinite relaxation \eqref{relax-mom}.
  Indeed, in its degree-$12$
 optimal solution $(\blambda^*,\bphi^*)$ obtained by running the 
 GloptiPoly software \cite{GloptiPoly} that implements the Moment-SOS hierarchy, $\bphi^*$ satisfies 
 the flatness condition \eqref{curto}, and the atoms can be extracted
 by a linear algebra subroutine. However we could notice that if we enlarge the set $S$, then 
  one needs to go to higher degrees in the hierarchy with potential numerical instabilities.
 \end{ex}

\section{The multivariate case}
\label{multivariate}
The result in the univariate case 
extends to the multivariate case with $\mu$ on $\R^d$, provided that the set of parameters
$(\m,\bSigma)\in S\subset\R^d\times \R^{d(d+1)/2}$ is a compact basic semi-algebraic set.
For  instance one may consider the case where $(\m,\bSigma)\in S$ with
\[S\,:=\,\{\,(\m,\bSigma): \:a\,\mathbf{I}\,\preceq \bSigma\,\preceq\,b\,\mathbf{I}\,;\quad g_j(\m)\,\geq\,0\,,\: j=1,\ldots,s\,\}\,,\]
for some polynomials $g_j\in\R[m_1,\ldots,m_d]$, $j=1,\ldots,s$, and
 some given scalars $0<a<b$.  Then using determinants of $\Sigma=(\sigma_{ij})_{i,j\leq d}$,
 the constraints $a\,\mathbf{I}\,\preceq \bSigma\,\preceq\,b\,\mathbf{I}$ reduces to
 $2d$ polynomials inequality constraints $q_k(\bsigma)\geq0$, $k=1,\ldots,2d$, with two of them
 of degree $d$. Then the set
 \begin{equation}
 \label{def-set-S}
 S\,=\,\{\,(\m,\bsigma)\,:\: g_j(\m)\,\geq\,0\,,\:j=1,\ldots,s\,;\: q_k(\bsigma)\,\geq\,0\,,\:k=1,\ldots,2d\,\}\subset\R^d\times\R^{d(d+1)/2}\,.\end{equation}
 As $S$ is compact and assuming one knows a scalar $R>0$ such that 
 \[R^2-\Vert\m\Vert^2-\Vert\bsigma\Vert^2\geq\,0\,,\quad\forall (\m,\bsigma)\,\in\, S\,,\]
 we may add the redundant quadratic constraint $R^2-\Vert\m\Vert^2-\Vert\bsigma\Vert^2\geq0$
 (relabelled as $g_1(\m,\bsigma)\geq0$)
 in the definition \eqref{def-set-S} of $S$ without changing $S$. In doing so, the quadratic module
 \begin{equation}
 \label{quad-module-multi}
 Q(g,q)\,=\,\{\,\sum_{j=0}^s\theta_j\,g_j+\sum_{k=1}^{2d}\theta'_k\,q_k\,:\: \theta_j,\theta'_k\in\Sigma[\m,\bsigma]\,\}\,\end{equation}
 is Archimedean. Next, as in the univariate case one introduces the polynomials $(p_{\balpha}\in\R[\m,\bsigma]_{\vert\balpha\vert})_{\balpha\in\N^d}$ defined by:
 \begin{equation}
 \label{def-pbalpha}
 p_{\balpha}(\m,\bsigma)\,:=\,\frac{1}{(2\pi)^{d/2}\sqrt{\mathrm{det}(\bSigma)}}\int \x^{\balpha}\exp(-(\x-\m)^T\bSigma^{-1}(\x-\m)/2)\,\d\x\,,
 \end{equation}
for every $\balpha\in\N^d$.
 Indeed every moment $\int \x^{\balpha}\,\d\nu$ of a Gaussian probability measure $\nu=\mathcal{N}(\m,\Sigma)$, is an explicit polynomial
 of its parameters $(\m,\bsigma)$, of total degree at most $j$. Moreover, the marginal of a Gaussian measure $\mu=\mathcal{N}(\m,\bSigma)$  with respect to $x_i$ is the Gaussian measure $\mathcal{N}(m_i,\Sigma_{ii})$. Therefore
  \begin{equation}
  \label{p2j-multi}
  \int x^{2j}_i\,\d\mu\,=\,p_{2j}(m_i,\Sigma_{ii})\,,\quad\forall j\in\N\,;\: i=1,\ldots,d\,,
  \end{equation}
 where  $p_{2j}$ has been defined in \eqref{lem1-1}.
 Next, if $\mu\in\mathscr{P}(\R^d)$,
 the multivariate analogue of \eqref{OT-1} reads:
 
\begin{eqnarray}
\nonumber
\tau&=&\displaystyle\inf_{\phi\in\mathscr{P}(S)}W_2(\mu,\nu_\phi)^2\\
&=&
\label{OT-multi-1}
\displaystyle\inf_{\phi\in\mathscr{P}(S),\lambda\in\mathscr{P}(\R^{2d})}\,\{\,\displaystyle\int \Vert \x-\y\Vert^2\,\d\lambda: \:\lambda_{\x}\,=\,\mu\,;\:\lambda_{\y}\,=\,\nu_\phi\}\,.
\end{eqnarray}
and the analogue of the moment  formulation \eqref{OT-moments} reads:
\begin{equation}
\label{OT-mult-2}
\begin{array}{rl}
\displaystyle\inf_{\phi\in\mathscr{P}(S),\lambda\in\mathscr{P}(\R^{2d})}&
\{\,\displaystyle\int \Vert \x-\y\Vert^2\,\d\lambda:\:
\lambda_{\balpha 0}\,=\,\mu_{\balpha}\,,\:\forall\balpha\in\N^d\,;\\
&
\lambda_{0\balpha}-\phi(p_{\balpha}(\m,\bsigma))\,=\,0\,,\quad\forall\,\balpha\in\N^d\,\}\,.
\end{array}
\end{equation}

\begin{ass}
\label{ass-1}
(i) The measure $\mu$ satisfies: $\sup_i\int \exp(c\,\vert x_i\vert)\,\d\mu <\infty$ for some $c>0$.

(ii) The set $S$ in \eqref{def-set-S} is compact with nonempty interior, and
the quadratic module \eqref{quad-module-multi} is Archimedean.
\end{ass}

\begin{thm}
\label{th-OT-multi}
Let Assumption \ref{ass-1} hold. Then:

(i) The optimal transport problem \eqref{OT-multi-1} has an optimal solution $(\phi^*,\lambda^*)\in\mathscr{P}(S)\times\mathscr{P}(\R^{2d})$ which is also an optimal solution of \eqref{OT-mult-2}. Moreover both measures
$\lambda^*\in\mathscr{P}(\R^{2d})$ and $\nu_{\phi^*}\in\mathscr{P}(\R^d)$ are moment determinate.

(ii) Moreover, $\tau=0$ if and only if $\lambda^*=\mu\otimes\mu
$ and $\mu=\nu_{\phi^*}$,  i.e., $\mu$ is a Gaussian mixture with $\phi^*$ a mixing measure of parameters 
$(\m,\bSigma)\in S$.
\end{thm}
\paragraph{Sketch of the proof}
As in the proof of Theorem \ref{th-OT} in the univariate case let $(\lambda^{(n)},\phi^{(n)})_{n\in\N}$ be a minimizing sequence of \eqref{OT-multi-1}.
As $S$ is compact there exists a subsequence $(n_k)_{k\in\N}$ and a probability measure $\phi^*\in\mathscr{P}(S)$
such that $\phi^{(n_k)}\Rightarrow\phi^*$ as $k\to\infty$. 

Let $d'=d+d(d+1)/2$ and recall that $S\subset\R^{d'}$. Following exactly the same steps as in the proof of Theorem \ref{th-OT}, there exists a subsequence
denoted $(n'_\ell)_{\ell\in\N}$  and an infinite sequence $\blambda^*=(\lambda^*_{\balpha})_{\balpha\in\N^{2d}}$, such that
\[\lim_{\ell\to\infty}\lambda^{(n'_\ell)}_{\balpha}\,=\,\lambda^*_{\balpha}\,,\quad\forall\balpha\in\N^{2d}\,;\quad
\lim_{\ell\to\infty}\phi^{(n'_\ell)}_{\bbeta}\,=\,\phi^*_{\bbeta}\,,\quad\forall\bbeta\in\N^{d'}\,.\]
Moreover as $S$ is compact and in view of \eqref{p2j-multi},
and \eqref{aux-2j-1}-\eqref{aux-2j-2},
and by Corollary \ref{coro1},
\[\sum_{j=1}^\infty \phi^*(p_{2j}(m_i,\Sigma_{ii}))^{-1/2j}\,=\,+\infty\,,\quad\forall i=1,\ldots,d\,.\]
and therefore
\[\sum_{j=1}^\infty \lambda^*(y_i^{2j})^{-1/2j}\,=\,+\infty\,,\quad\forall i=1,\ldots,d\,.\]
Next, by Assumption \ref{ass-1}(i) on $\mu$, one also has
\[\sum_{j=1}^\infty \lambda^*(x_i^{2j})^{-1/2j}\,=\,\sum_{j=1}^\infty \mu(x_i^{2j})^{-1/2j}\,=\,
+\infty\,,\quad\forall i=1,\ldots,d\,,\]
and therefore the moment sequence $\blambda^*$ satisfies multivariate Carleman's condition
(see e.g. \cite[Proposition 2.37]{CUP-2}), which in turn implies 
that it is the moment sequence of some measure $\lambda^*\in\mathscr{M}(\R^{2d})_+$ which is moment determinate.
Then again as in the proof of Theorem \ref{th-OT} we may conclude that $(\lambda^*,\phi^*)$ is an optimal solution of 
\eqref{th-OT-multi}.\qed \\

Next, let $d_j=\lceil\mathrm{deg}(g_j)/2\rceil$ and $t_k=\lceil\mathrm{deg}(q_k)/2\rceil$, for all $j$ and $k$.
Then for every $n\geq n_0=\max_{j,k}[ \,d_j\,,\,t_k\,]$,
the multivariate analogue of the semidefinite relaxation \eqref{relax-mom} 
reads:
\begin{equation}
\label{OT-mult-3}
\begin{array}{rl}
\tau_n\,=\,\displaystyle\inf_{\bphi\,,\blambda}\,\{\,\displaystyle\int \Vert \x-\y\Vert^2\,d\lambda:&
\lambda_{\balpha,0}\,=\,\mu_{\balpha}\,;\:
\lambda_{0,\balpha}-\phi(p_{\balpha}(\m,\bsigma))\,=\,0\,,\quad\forall\,\balpha\in\N^d_{2n}\,;\\
&\M_n(\blambda)\,,\:\M_n(\bphi)\,\succeq0\,;\\
&\M_{n-d_j}(g_j\cdot\bphi)\,,\:\M_{n-t_k}(q_k\cdot\bphi)\succeq0\,;\\
&j=1,\ldots,s\,;\:k=1,\ldots,2d\:\}\,.
\end{array}
\end{equation}
Then an analogue of Theorem \ref{th-hierarchy} holds and its proof is along the same lines.
Also Curto \& Fialkow's flatness condition \cite[Theorem 2.47]{CUP-2} is also valid in the multivariate setting.
Similarly  there is an exact analogue of Theorem \ref{mu-mixture}.

\section{Conclusion}

We have considered Gaussian mixtures (with parameters $(m,\sigma)$  in a given compact set $S$) 
closest in Wasserstein distance, to a given measure $\mu$. Such Gaussian mixtures are
optimal solutions of an infinite-dimensional optimal-transport linear program (LP)
in which one marginal constraint contains the unknown  mixing measure. Non-uniqueness
is related to a classical identifiably issue. This LP
can be solved by the Moment-SOS hierarchy, i.e.,  a sequence of semidefinite programs (convex relaxations) whose size increases with the number of moment constraints considered. That $\mu$ cannot be a Gaussian mi	omic mixing measure on $S$ with finite support, a latter can sometimes be extracted from an optimal solution at some step of the hierarchy. In addition to the identifiability issue, an interesting research direction is concerned with whether 
a similar approach can be  implemented when the distance is now measured in total variation instead of Wasserstein.

 \section{Appendix}
 \label{appendix}
 In this paper we mainly use the $W_2(\mu,\nu)$-optimal transport problem  \eqref{def-distance-0}
 for two probability measures $\mu$ and $\nu$, but we could also use the
 $W_1(\mu,\nu)$-optimal transport problem. Its  primal formulation reads
 \[W_1(\mu,\nu)\,=\,\displaystyle\inf_{\lambda\in\mathscr{P}(\R^2)}\,\{\,\int_{\R^2}\vert x-y\vert\,\d\lambda(x,y):\:
 \lambda_x\,=\,\mu\,;\:\lambda_y\,=\,\nu\,\},\]
 while its dual formulation reads
 \[W_1(\mu,\nu)\,=\,\displaystyle\sup_{f,g}\,\{\,\int_{\R^2}f(x)\,\d\mu(x)+\int g(y)\,\d\nu(y):\:
 f(x)+g(y)\,\leq\,\vert x-y\vert\,,\quad \forall x,y\,\in\R\,\}\,.\]
 In order to proceed in a manner similar as for the $W_2$-distance, we need to 
 write $\R^2=\overline{X_1\cup X_2}$ with $X_1:=\{(x,y):x < y\}$ and $X_2:=\{(x,y): x> y\}$, and impose $\lambda=\lambda_1+\lambda_2$  with $\mathrm{supp}(\lambda_1)=\overline{X_1}$ and 
 $\mathrm{supp}(\lambda_2)=\overline{X_2}$.
 \subsection{Proof of Theorem \ref{th-OT}}
 \begin{proof}
(i)  Let $(\lambda^{(n)},\phi^{(n)})_{n\in\N}\subset \mathscr{P}(\R^2)\times\mathscr{P}(S)$ be a minimizing sequence of \eqref{OT-1} with $\rho_n:=W_2(\mu-\nu_{\phi^{(n)}}) 
 \downarrow \tau$ as $n$ increases. As $S$ is compact, the sequence $(\phi^{(n)})_{n\in\N}$ is tight and by Prohorov's theorem, there exists a subsequence $(n_k)_{k\in\N}$ and a probability measure $\phi^*\in\mathscr{P}(S)$ such that
 \[\lim_{k\to\infty}\int h\,\d\phi^{(n_k)}\,=\,\int h\,\d\phi^*\,,\quad\forall h\in\mathscr{C}(S)\quad\mbox{[denoted $\phi_{n_k}\Rightarrow\phi^*$]}\,.\]
 In particular, $\phi^{(n_k)}_{(i,j)}\to \phi^*_{(i,j)}$ for all $(i,j)\in\N^2$. In addition, as $p_j\in\R[m,\sigma]$,
 \begin{equation}
 \label{aux-1}
 \lim_{k\to\infty}\lambda^{(n_k)}_{(0,j)}\,=\,
  \lim_{k\to\infty}\int p_j\,d\phi^{(n_k)}
 \,=\,\int p_j\,d\phi^*\,,\quad\forall j\in\N\,,\end{equation}
 and by feasibility, we also have 
 \[\lim_{k\to\infty}\lambda^{(n_k)}_{(j,0)}\,=\,\mu_j\,,\quad\forall j\in\N\,.\]
 We want to prove that
  \[\forall i,j\in\N\,:\quad \lim_{k\to\infty}\lambda^{(n_k)}_{(i,j)}\,=\,\int x^i y^i\,\d\lambda^*(x,y)\,,\]
  for some determinate measure $\lambda^*$ on $\R^2$. That is,
 the vector of moments $\blambda^{(n_k)}$ converges to the vector of moments of 
 $\lambda^*$, and in particular 
 \[\lambda^*_{(j,0)}\,=\,\mu_j\,\quad \forall j\in\N\,;\quad\lambda^*_{(0,j)}\,=\,\int_Sp_j\,\d\phi^*\,=\,\int_{\R}x^j\,\d\nu_{\phi^*},\quad \forall j\in\N\,.\]
Notice that then  $(\lambda^*,\phi^*)$ is an optimal solution of \eqref{OT-1}.
 
 As $S$ is compact, $\vert m\vert<M$ and $\sigma<M$ for some $M>0$ and therefore by \eqref{aux-2j-1},
 \[\lambda^{(n)}_{(0,2j)}=\,\phi^{(n)}(p_{2j})\quad\Rightarrow
 \lambda^{(n)}_{(0,2j)}\,=\,\int_S p_{2j}\,\d\phi^{(n)}
  \,<\,(2Mj)^{2j}  \,=:\,\rho_j,\quad \forall j\in\N\,.\]
 This combined with $\lambda^{(n)}_{(2j,0)}=\mu_{2j}$ yields that the moment matrix $\M_k(\blambda^{(n)})$ of the moment sequence $\blambda^{(n)}$ of the measure $\lambda^{(n)}$ satisfies $\M_k(\blambda^{(n)})\succeq0$ for every $k$, and
 
 \[\forall (k,\ell)\in\N^2\mbox{ with $k+\ell\leq 2j$:}\quad \vert\lambda^{(n)}_{k,\ell}\vert\,\leq \,\max[1,\mu_{2j},\rho_j\,]\,=:\,\rho'_j,\quad\forall j\in N\,.\]
 See \cite[Proposition 3.6, p. 60]{CUP}.
 Then define the new infinite sequence $\hat{\blambda}^{(n)}$ 
 by
\begin{equation}
\label{scaling-1}
\begin{array}{rcl}
\hat{\lambda}^{(n)}_{(i,j)}&:=&\lambda_{(i,j)}/\rho'_k\,,\quad\forall (i,j)\mbox{ with }2k<i+j\leq 2k\,,\quad k=1,\ldots,\\
\end{array}
\end{equation}
so that $\hat{\blambda}^{(n)}$ is an element of the unit ball of $\ell_\infty$, the Banach space of (uniformly) bounded sequences.
As the unit ball $\B_{\ell_\infty}(0,1)$ of $\ell_\infty$ is 
sequentially compact in the weak-star topology $\sigma(\ell_\infty,\ell_1)$, there is a subsequence $(n'_\ell)_{\ell\in\N}\subset (n_k)_{k\in\N}$ and an infinite vector $\hat{\blambda}^*\in \B_{\ell_\infty}(0,1)$ such that 
(in particular)
\[\lim_{\ell\to\infty}\hat{\lambda}^{(n'_\ell)}_{(i,j)}\,=\,
\hat{\lambda}^*_{(i,j)}\,,\quad\forall (i,j)\in\N^2\,.\]
Then by the reverse scaling of \eqref{scaling-1} for 
$\hat{\blambda}^*$ 
\begin{equation}
\label{limit-1}
\lim_{\ell\to\infty}\lambda^{(n'_\ell)}_{(i,j)}\,=\,
\lambda^*_{(i,j)}\,;\quad\forall (i,j)\in\N^2\,,\end{equation}
for some infinite vector $\blambda^*=(\lambda^*_{(i,j)})_{(i,j)\in\N^2}$.
In addition, by \eqref{limit-1}, $\M_n(\blambda^*)\succeq0$ for all $n\in\N$, and
\[\lambda^*_{(j,0)}\,=\,\mu_j\,;\quad\lambda^*_{(0,j)}\,=\,
\phi^*(p_j)\,,\quad\forall j\in\N\,,\]
and  by Corollary \ref{coro1},
\[\sum_{j=1}^\infty (\lambda^*_{(0,2j)})^{-1/2j}\,=\,+\infty\,.\]
As $\lambda^*_{(2j,0)}=\mu_{2j}$ for all $j\in\N$, 
and $\mu$ satisfies Carleman's condition, 
then by Theorem \ref{th-carleman},
$\blambda^*$ has a representing measure $\lambda^*$ on $\R^2$, which is moment determinate. This implies that $(\lambda^*,\phi^*)$ is a feasible solution of \eqref{OT-1}.
Finally, as $(\lambda^{(n'_\ell)},\phi^{(n'_\ell)})$ is a minimizing sequence of \eqref{OT-1}, then by \eqref{limit-1},
\[\tau\,=\,\lim_{\ell\to\infty}\rho_{n_\ell}\,=\,\lim_{\ell\to\infty}\int (x-y)^2\,d\lambda^{(n'_\ell)}\,=\,
 \int (x-y)^2\,d\lambda^*\quad\mbox{[by  \eqref{limit-1}]}\,,\]
which shows that
$(\lambda^*,\phi^*)$  an optimal solution of \eqref{OT-1}. 

Finally, in what precedes we have only used the respective moments $\blambda^{(n)}$ and $\bphi^{(n)}$ of the measures
$\lambda^{(n)}$ and $\phi^{(n)}$, and the constraints of \eqref{OT-moments}. Hence by considering a minimizing sequence  $(\lambda^{(n)},\phi^{(n)})$ of \eqref{OT-moments}
instead of \eqref{OT-1}, one reaches the same conclusion.

(ii) \emph{If part:} Straightforward. Indeed if $\mu$ is a Gaussian mixture with $\phi^*$ a mixing measure of parameters $(m,\sigma)\in S$ then $\mu=\nu_{\phi^*}$ and with $\lambda^*:=\mu\otimes\mu$,
the couple $(\lambda^*,\phi^*)$ is a feasible solution of \eqref{OT-1} with value $\tau=0$, hence an optimal solution of \eqref{OT-1}.

\emph{Only if part:} By (i) let $(\lambda^*,\phi^*)$ be an optimal solution of \eqref{OT-1}. As $0=\tau=\int (x-y)^2d\lambda^*$,
it follows that $\mathrm{supp}(\lambda^*)\subset\{(x,x): x\in\R\}$, and therefore $\lambda^*_x=\lambda^*_y$, i.e., $\lambda^*=\mu\otimes\mu$, and therefore
as $\lambda^*_y=\nu_{\phi^*}$, one obtains $\mu=\nu_{\phi^*}$,
the desired result.
\end{proof}

 \subsection{Proof of Theorem \ref{th-hierarchy}}
 \begin{proof}
 (i) Let $(\blambda,\bphi)$ be a feasible solution of \eqref{relax-mom}. As $g_1(m,\sigma)=R^2-m^2-\sigma^2$, the constraint $\M_{n-1}(g_1\cdot\phi)\succeq0$, implies that
\[\phi(\sigma^{2n})\,<\,R^{2n}\,;\quad \phi(m^{2n})\,<\,R^{2n}\,.\]
By \cite[Proposition 3.6, p. 60]{CUP}, this combined with $\M_n(\bphi)\succeq0$, and $\phi_{(0,0)}=1$, yields 
$\vert\phi_{(i,j)}\vert<\max[1,R^{2n}]$ for all $i,j$ with $i+j\leq 2n$. 
Moreover, $\lambda_{2n,0}=\mu_{2n}$, and by \eqref{aux-2j-1},
\[\lambda_{(0,2n)}\,=\,\phi(p_{2n})\,<\,  (2nR)^{2n}\,=:\,\rho_n\,.\]
Again by \cite[Proposition 3.6, p. 60]{CUP}, for all $(i,j)$ with $i+j\leq 2n$,
\[\vert\,\lambda_{(i,j)}\,\vert\,<\,\max[1,\lambda_{(2n,0)},\lambda_{(0,2n)}]\,<\,\max[1,\mu_{2n},\rho_n]\,=:\,\rho'_n\,,\]
which implies that the feasible set of \eqref{relax-mom} in compact, and therefore \eqref{relax-mom} has an optimal solution $(\blambda^{(n)},\bphi^{(n)})$ with value $\lambda^{(n)}((x-y)^2)=\tau_n$, and as \eqref{relax-mom} is a relaxation of \eqref{OT-moments}, $0\leq \tau_n\leq\tau$ for all $n$.

(ii) Complete the finite vector $\blambda^{(n)}$ (resp. $\bphi^{(n)}$) with zeros to make it an infinite sequence $\blambda^{(n)}=(\lambda^{(n)}_{(i,j)})_{(i,j)\in\N^2}$
(resp. $\bphi^{(n)}=(\bphi^{(n)}_{(i,j)})_{(i,j)\in\N^2}$).
Then define the new infinite sequences $\hat{\blambda}^{(n)}$ and
$\hat{\bphi}^{(n)}$ by
\begin{equation}
\label{scaling}
\begin{array}{rcl}
\hat{\lambda}^{(n)}_{(i,j)}&:=&\lambda_{(i,j)}/\rho_k\,,\quad\forall (i,j)\mbox{ with }2k<i+j\leq 2k\,,\quad k=1,\ldots\\
\hat{\phi}^{(n)}_{(i,j)}&:=&\phi_{(i,j)}/R^{2k}\,,\quad\forall (i,j)\mbox{ with }2k<i+j\leq 2k\,,\quad k=1,\ldots,
\end{array}
\end{equation}
so that $\hat{\blambda}^{(n)}$ is an element of the unit ball of $\ell_\infty$, the Banach space of (uniformly) bounded sequences, and smililarly for $\hat{\bphi}^{(n)}$. Again, as the unit ball of $\ell_\infty$ is 
sequentially compact in the weak-star topology $\sigma(\ell_\infty,\ell_1)$, there is a subsequence $(n_k)_{k\in\N}$ and infinite vectors $\hat{\blambda}^*\in\ell_\infty$ and $\hat{\bphi}^*\in\ell_\infty$ such that
\[\lim_{k\to\infty}\hat{\lambda}^{(n_k)}_{(i,j)}\,=\,
\hat{\lambda}^*_{(i,j)}\,;\quad
\lim_{k\to\infty}\hat{\phi}^{(n_k)}_{(i,j)}\,=\,\hat{\phi}^*_{(i,j)}\,,\quad \forall (i,j)\in\N^2\,.\]
Then by the reverse scaling of \eqref{scaling} for $\hat{\blambda}^*$ and $\hat{\bphi}^*$,
\begin{equation}
\label{limit}
\lim_{k\to\infty}\lambda^{(n_k)}_{(i,j)}\,=\,
\lambda^*_{(i,j)}\,;\quad
\lim_{k\to\infty}\phi^{(n_k)}_{(i,j)}\,=\,\phi^*_{(i,j)}\,,\quad \forall (i,j)\in\N^2\,,\end{equation}
for some infinite vectors $\blambda^*$ and $\bphi^*$.
Next, by \eqref{limit}, $\M_n(\blambda^*)\succeq0$,
$\M_n(\bphi^*)\succeq0$, and $\M_n(g_j\cdot\bphi^*)\succeq0$ for all $n$, with
\[\lambda^*_{(j,0)}\,=\,\mu_j\quad\mbox{and}\quad
\lambda^*_{(0,j)}\,=\,\phi^*(p_j)\,,\quad\forall j\,\in\N\,.\]
As $g_1(m,\sigma)=R^2-m^2-\sigma^2$, the quadratic module
\[Q(g)\,=\,\{\,\sum_{j=0}^s \theta_j(m,\sigma)\,g_j(m,\sigma)\::\: \theta_j\in\Sigma[m,\sigma]\,\}\]
is Archimedean and therefore, by Putinar's Positivstellensatz \cite{putinar}, $\bphi^*$ has a representing measure on $S$.
Moreover, as in the proof of Theorem \ref{th-OT}, by Corollary \ref{coro1},
\[\sum_{j=1}^\infty (\lambda^*_{(0,2j)})^{-1/2j}\,=\,+\infty\,,\]
and as $\lambda^*_{(2j,0)}=\mu_{2j}$ for all $j\in\N$, 
and $\mu$ satisfies Carleman's condition, 
then by Theorem \ref{th-carleman},
$\blambda^*$ has a representing measure $\lambda^*$ on $\R^2$, which is moment determinate. 
In particular its marginal $\lambda^*_y$ with moments $(\lambda^*_{(0,j)})_{j\in\N}$ is also moment determinate.
Next, let $\nu_{\phi^*}$ be the measure on $\R$ with Gaussian mixture $\phi^*$. As
\[\lambda^*_{(0,j)}\,=\,\phi^*(p_j)\,=\,\int x^j\,\d\nu_{\phi^*}(x)\,,\quad\forall j\in\N\,,\]
and as $\lambda^*_y$ is moment determinate, this show that $\lambda^*_y=\nu_{\phi^*}$.
Hence $(\phi^*,\lambda^*)$ is feasible for \eqref{OT-1} with value $\lambda^*((x-y)^2)$. In addition, 
as $\tau_{n}\leq\tau$ for all $n$,
\[\tau\,\leq\,\lambda^*(x-y)^2\,=\,\lim_{\ell\to\infty}\lambda^{(n'_\ell)}((x-y)^2)\quad\mbox{([by \eqref{limit}])}\quad\,=\,
\lim_{\ell\to\infty}\tau_{n'_\ell}\,\leq\,\tau\,,\]
so that $\tau=\lambda^*((x-y)^2)$, and therefore $(\lambda^*,\phi^*)$ is an optimal solution of \eqref{OT-1} (and of \eqref{OT-moments} as well).
\end{proof}

 \bibliographystyle{amsplain}
\bibliography{lasserrebib}

\end{document}